\documentclass{amsart}
\usepackage{amssymb}

  \theoremstyle{definition}
  \newtheorem{definition}             {Definition}
  
  \newtheorem{remark}     [definition]{Remark}

  \theoremstyle{plain}
  \newtheorem{lemma}      [definition]{Lemma}
  \newtheorem{proposition}[definition]{Proposition}
  \newtheorem{theorem}    [definition]{Theorem}
  \newtheorem{corollary}  [definition]{Corollary}
  \newtheorem{conjecture}             {Conjecture}

  \newcommand{\opleft}{\{}
  \newcommand{\opright}{\}}

\begin{document}

\title[Polynomial identities for ternary intermolecular recombination]
{Polynomial identities for ternary intermolecular recombination}

\author{Murray R. Bremner}

\address{Department of Mathematics and Statistics, University of Saskatchewan,
Canada}

\email{bremner@math.usask.ca}

\begin{abstract}
The operation of binary intermolecular recombination, originating in the theory
of DNA computing, permits a natural generalization to $n$-ary operations which
perform simultaneous recombination of $n$ molecules. In the case $n = 3$, we
use computer algebra to determine the polynomial identities of degree $\le 9$
satisfied by this trilinear nonassociative operation. Our approach requires
computing a basis for the nullspace of a large integer matrix, and for this we
compare two methods: ($i$) the row canonical form, and ($ii$) the Hermite
normal form with lattice basis reduction. In the conclusion, we formulate some
conjectures for the general case of $n$-ary intermolecular recombination.
\end{abstract}

\maketitle

\section{Introduction}

Let $X$ be a finite set, and let $X^\ast$ be the free semigroup on $X$; thus
$X^\ast$ consists of all finite nonempty strings of elements of $X$,
  \[
  X^\ast
  =
  \{ \, x_1 x_2 \cdots x_k \,|\, x_1, x_2, \hdots, x_k \in X \, (k \ge 1) \, \},
  \]
together with the binary operation of concatenation,
  \[
  ( x_1 x_2 \cdots x_k )( x_{k+1} x_{k+2} \cdots x_{k+\ell} )
  \longmapsto
  x_1 x_2 \cdots x_{k+\ell}.
  \]
Let $n$ be a positive integer, and let $( X^\ast )^n$ be the set of all
$n$-tuples of elements of $X^\ast$. Let $R(X,n)$ be the vector space with basis
$( X^\ast )^n$ over the field $\mathbb{Q}$ of rational numbers.

\begin{definition} \label{definitionrecombination}
On $R(X,n)$ we introduce a multilinear nonassociative operation
$\opleft a_1, a_2, \hdots, a_n \opright$, called
\textbf{$n$-ary intermolecular recombination},
and defined on basis elements as follows: For
$a_i = ( a_{i1}, a_{i2}, \hdots, a_{in} ) \in ( X^\ast )^n$ ($i = 1, \hdots, n$)
we set
  \begin{equation}\label{operation}
  \opleft a_1, a_2, \hdots, a_n \opright
  =
  \sum_{\sigma \in S_n}
  ( a_{\sigma(1),1}, a_{\sigma(2),2}, \hdots, a_{\sigma(n),n} ),
  \end{equation}
where $S_n$ is the symmetric group on $\{1,2,\hdots,n \}$.
\end{definition}

In Definition \ref{definitionrecombination}, each $a_i$ can be
regarded as the abstract representation of a molecule divided into $n$
consecutive submolecules $a_{i1}$, $a_{i2}$, $\hdots$, $a_{in}$. For example,
if $X = \{ A, C, G, T \}$ represents the four bases found in DNA, then each
$a_i$ represents a DNA sequence partitioned into $n$ subsequences. With this
interpretation, the sum defining $\opleft a_1,a_2,\hdots,a_n \opright$
expresses the results of recombining the $n$ molecules in all possible ways
which preserve the positions of the $n$ submolecules in each molecule.
Formula \eqref{operation} is very similar to the definition of the permanent of an
$n \times n$ matrix; see Wanless \cite{Wanless}.

The next result  is clear.

\begin{lemma}
The operation of $n$-ary intermolecular recombination in Definition
\ref{definitionrecombination}
satisfies the \textbf{complete symmetry identity} in degree $n$:
  \begin{equation} \label{completesymmetry}
  \opleft a_{\pi(1)}, a_{\pi(2)}, \hdots, a_{\pi(n)} \opright
  =
  \opleft a_1, a_2, \hdots, a_n \opright,
  \end{equation}
for every $\pi \in S_n$; thus $\opleft a_1, a_2, \hdots, a_n \opright$ depends
only on the multiset of arguments, not the order of the arguments.
\end{lemma}

In the case $n = 2$, the operation of binary intermolecular recombination was
introduced by Landweber and Kari \cite{LandweberKari}.

Bremner \cite{Bremner} used computer algebra to establish the following proposition.

\begin{proposition}
Every polynomial identity of degree $\le 5$ satisfied by binary intermolecular
recombination is a consequence of commutativity (binary complete symmetry) and
the \textbf{binary recombination identity}
  \begin{equation}
  \label{recombination}
  \begin{array}{l}
  2 \opleft  \opleft  \opleft  a, b  \opright, c  \opright, d  \opright
  - \opleft  \opleft  \opleft  a, b  \opright, d  \opright, c  \opright
  - \opleft  \opleft  \opleft  a, c  \opright, b  \opright, d  \opright
  - \opleft  \opleft  \opleft  b, c  \opright, a  \opright, d  \opright
  \\[4pt]
  {}
  + \opleft  \opleft  a, b  \opright, \opleft  c, d  \opright  \opright
  = 0.
  \end{array}
  \end{equation}
\end{proposition}

Sverchkov \cite{Sverchkov} has recently proved the following theorem.

\begin{theorem}
Every identity, with no restriction on the degree, satisfied by binary intermolecular
recombination is a consequence of commutativity and the binary recombination identity.
\end{theorem}

In related work, Bremner, Piao and Richards \cite{BremnerPiaoRichards} have shown
that every polynomial identity satisfied by the zygotic algebras of simple
Mendelian inheritance is a consequence of commutativity and the binary
recombination identity.

In the present paper we consider the case $n = 3$; we use computer algebra to
determine the polynomial identities of degree $\le 9$ satisfied by ternary
intermolecular recombination. A monomial which involves $d$ applications of an
$n$-ary operation has degree $d(n{-}1) + 1$; thus polynomial identities for
ternary intermolecular recombination exist only in odd degrees.

\section{Binary intermolecular recombination}

In this section we recall the results of Bremner \cite{Bremner} on polynomial
identities for binary intermolecular recombination, and show how these results
may also be obtained (and slightly improved) using the Hermite normal form with
lattice basis reduction.

In degree 4, there are 15 commutative nonassociative multilinear monomials, 12
with association type $\opleft \opleft \opleft a, b \opright, c \opright, d
\opright$ and 3 with association type $\opleft \opleft a, b \opright, \opleft
c, d \opright \opright$. We order these monomials first by association type,
and then lexicographically by the permutation of the variables:
  \begin{alignat*}{4}
  &\opleft \opleft \opleft a, b \opright, c \opright, d \opright, &\quad
  &\opleft \opleft \opleft a, b \opright, d \opright, c \opright, &\quad
  &\opleft \opleft \opleft a, c \opright, b \opright, d \opright, &\quad
  &\opleft \opleft \opleft a, c \opright, d \opright, b \opright,
  \\
  &\opleft \opleft \opleft a, d \opright, b \opright, c \opright, &\quad
  &\opleft \opleft \opleft a, d \opright, c \opright, b \opright, &\quad
  &\opleft \opleft \opleft b, c \opright, a \opright, d \opright, &\quad
  &\opleft \opleft \opleft b, c \opright, d \opright, a \opright,
  \\
  &\opleft \opleft \opleft b, d \opright, a \opright, c \opright, &\quad
  &\opleft \opleft \opleft b, d \opright, c \opright, a \opright, &\quad
  &\opleft \opleft \opleft c, d \opright, a \opright, b \opright, &\quad
  &\opleft \opleft \opleft c, d \opright, b \opright, a \opright,
  \\
  &\opleft \opleft a, b \opright, \opleft c, d \opright \opright, &\quad
  &\opleft \opleft a, c \opright, \opleft b, d \opright \opright, &\quad
  &\opleft \opleft a, d \opright, \opleft b, c \opright \opright.
  \end{alignat*}
We expand each monomial by three applications of binary intermolecular
recombination. For example, writing
  \[
  a = ( a_1, a_2 ),
  \quad
  b = ( b_1, b_2 ),
  \quad
  c = ( c_1, c_2 ),
  \quad
  d = ( d_1, d_2 ),
  \]
we obtain
  \begin{align*}
  \opleft \opleft \opleft a, b \opright, c \opright, d \opright
  &=
  ( a_1, d_2 ) + ( b_1, d_2 ) + 2 ( c_1, d_2 ) + ( d_1, a_2 ) + ( d_1, b_2 ) + 2 (d_1, c_2 ),
  \\
  \opleft \opleft a, b \opright, \opleft c, d \opright \opright
  &=
  ( a_1, c_2 ) + ( a_1, d_2 ) + ( b_1, c_2 ) + ( b_1, d_2 ) + ( c_1, a_2 ) + ( c_1, b_2 )
  \\
  &\quad
  + ( d_1, a_2 ) + ( d_1, b_2 ).
  \end{align*}
Each expansion is a linear combination (allowing zero coefficients) of 12
pairs, which we order lexicographically:
  \begin{alignat*}{6}
  &( a_1, b_2 ), &\quad
  &( a_1, c_2 ), &\quad
  &( a_1, d_2 ), &\quad
  &( b_1, a_2 ), &\quad
  &( b_1, c_2 ), &\quad
  &( b_1, d_2 ),
  \\
  &( c_1, a_2 ), &\quad
  &( c_1, b_2 ), &\quad
  &( c_1, d_2 ), &\quad
  &( d_1, a_2 ), &\quad
  &( d_1, b_2 ), &\quad
  &( d_1, c_2 ).
  \end{alignat*}
We store the expansions in the $12 \times 15$ matrix $E$ in which entry $(i,j)$
contains the coefficient of pair $i$ in the expansion of monomial $j$; see
Table \ref{binaryexpansionmatrix}. The coefficient vectors of the polynomial
identities satisfied by binary intermolecular recombination in degree 4 are the
nontrivial linear combinations of the columns of $E$ that give the zero vector;
that is, the nonzero vectors in the nullspace of $E$.

  \begin{table}
  \begin{center}
  \[
  \left[
  \begin{array}{rrrrrrrrrrrrrrr}
  0 & 0 & 0 & 1 & 0 & 1 & 0 & 1 & 0 & 1 & 2 & 2 & 0 & 1 & 1 \\
  0 & 1 & 0 & 0 & 1 & 0 & 0 & 1 & 2 & 2 & 0 & 1 & 1 & 0 & 1 \\
  1 & 0 & 1 & 0 & 0 & 0 & 2 & 2 & 0 & 1 & 0 & 1 & 1 & 1 & 0 \\
  0 & 0 & 0 & 1 & 0 & 1 & 0 & 1 & 0 & 1 & 2 & 2 & 0 & 1 & 1 \\
  0 & 1 & 0 & 1 & 2 & 2 & 0 & 0 & 1 & 0 & 1 & 0 & 1 & 1 & 0 \\
  1 & 0 & 2 & 2 & 0 & 1 & 1 & 0 & 0 & 0 & 1 & 0 & 1 & 0 & 1 \\
  0 & 1 & 0 & 0 & 1 & 0 & 0 & 1 & 2 & 2 & 0 & 1 & 1 & 0 & 1 \\
  0 & 1 & 0 & 1 & 2 & 2 & 0 & 0 & 1 & 0 & 1 & 0 & 1 & 1 & 0 \\
  2 & 2 & 1 & 0 & 1 & 0 & 1 & 0 & 1 & 0 & 0 & 0 & 0 & 1 & 1 \\
  1 & 0 & 1 & 0 & 0 & 0 & 2 & 2 & 0 & 1 & 0 & 1 & 1 & 1 & 0 \\
  1 & 0 & 2 & 2 & 0 & 1 & 1 & 0 & 0 & 0 & 1 & 0 & 1 & 0 & 1 \\
  2 & 2 & 1 & 0 & 1 & 0 & 1 & 0 & 1 & 0 & 0 & 0 & 0 & 1 & 1
  \end{array}
  \right]
  \]
  \end{center}
  \caption{The expansion matrix $E$ in degree 4 for $n = 2$}
  \label{binaryexpansionmatrix}
  \end{table}

There are two ways to compute a basis for the nullspace of a matrix $E$ with
entries in the ring $\mathbb{Z}$ of integers:
  \begin{enumerate}
  \item[(a)]
Compute the row canonical form (RCF) of $E$ over $\mathbb{Q}$, set the free
variables equal to the standard basis vectors in $\mathbb{Q}^d$ where $d$ is the
dimension of the nullspace, and solve for the leading variables; then (if necessary)
multiply each vector by the LCM of the denominators of its components to obtain
integer components and divide each new vector by the GCD of its (now integral)
components.
  \item[(b)]
Compute the Hermite normal form $H$ of the transpose $E^t$ together with a
transform matrix $U$ for which $U E^t = H$. The bottom $d$ rows of $U$ form a
basis for the integer nullspace lattice of $E$. Then use the LLL algorithm for
lattice basis reduction on the bottom $d$ rows of $U$. (This method was first
applied to polynomial identities by Bremner and Peresi \cite{BremnerPeresi}.)
  \end{enumerate}

  \begin{table}
  \begin{center}
  \[
  \left[
  \begin{array}{rrrrrrrrrrrrrrr}
  1 & 0 & 0 & 0 & 0 &  1 & 0 & -3 & -3 & -5 &  0 & -3 & -2 &  1 & -2 \\
  0 & 1 & 0 & 0 & 0 & -1 & 0 &  3 &  3 &  5 &  1 &  4 &  1 &  0 &  3 \\
  0 & 0 & 1 & 0 & 0 & -1 & 0 & -1 &  1 &  0 & -2 & -2 &  1 & -2 &  0 \\
  0 & 0 & 0 & 1 & 0 &  1 & 0 &  1 &  0 &  1 &  2 &  2 &  0 &  1 &  1 \\
  0 & 0 & 0 & 0 & 1 &  1 & 0 & -2 & -1 & -3 & -1 & -3 &  0 &  0 & -2 \\
  0 & 0 & 0 & 0 & 0 &  0 & 1 &  3 &  1 &  3 &  1 &  3 &  1 &  1 &  1
  \end{array}
  \right]
  \]
  \end{center}
  \caption{The row canonical form of the matrix $E$ for $n = 2$}
  \label{binaryrcf}
  \begin{center}
  \[
  \left[
  \begin{array}{rrrrrrrrrrrrrrr}
  -1 &  1 &  1 & -1 & -1 & 1 &  0 & 0 & 0 & 0 & 0 & 0 & 0 & 0 & 0 \\
  -1 &  0 &  2 & -1 &  0 & 0 & -1 & 0 & 0 & 0 & 0 & 0 & 0 & 1 & 0 \\
   2 & -1 & -1 &  0 &  0 & 0 & -1 & 0 & 0 & 0 & 0 & 0 & 1 & 0 & 0 \\
   0 & -1 &  2 & -2 &  1 & 0 & -1 & 0 & 0 & 0 & 1 & 0 & 0 & 0 & 0 \\
   2 & -3 &  0 & -1 &  2 & 0 & -1 & 0 & 0 & 0 & 0 & 0 & 0 & 0 & 1 \\
   3 & -3 & -1 &  0 &  1 & 0 & -1 & 0 & 1 & 0 & 0 & 0 & 0 & 0 & 0 \\
   3 & -3 &  1 & -1 &  2 & 0 & -3 & 1 & 0 & 0 & 0 & 0 & 0 & 0 & 0 \\
   3 & -4 &  2 & -2 &  3 & 0 & -3 & 0 & 0 & 0 & 0 & 1 & 0 & 0 & 0 \\
   5 & -5 &  0 & -1 &  3 & 0 & -3 & 0 & 0 & 1 & 0 & 0 & 0 & 0 & 0
  \end{array}
  \right]
  \]
  \end{center}
  \caption{The canonical basis of the nullspace of $E$ for $n = 2$}
  \label{binarybasis}
  \end{table}

For the matrix $E$ of Table \ref{binaryexpansionmatrix}, method (a) gives the
following results. The RCF over $\mathbb{Q}$ is given in Table \ref{binaryrcf}, and
the corresponding nullspace basis vectors (sorted by increasing length) are given in
Table \ref{binarybasis}. The squared Euclidean norms of these basis vectors are 6, 8, 8,
12, 20, 22, 34, 52, 70.

For each basis vector, we apply all permutations of $a, b, c, d$ to
the corresponding identity $I$ and store the results in a $24 \times 15$ matrix
$M(I)$; the rank of $M(I)$ is the dimension of the $S_4$-module of identities
which are consequences of $I$. Rows 2 and 3 are the shortest vectors (squared
norm 8) for which the corresponding identities generate the entire nullspace,
and row 3 corresponds to the binary recombination identity.

Method (a) provides a basis of integral vectors for the rational nullspace of
the matrix $E$. The disadvantage of this method is that this basis of integral vectors
may not be a integer basis for the nullspace lattice of $E$; see Example 17 in
Bremner and Peresi \cite{BremnerPeresi}.

\begin{definition}
The \textbf{nullspace lattice} of the $m \times n$ integer matrix $E$ is
  \[
  L(E) = \{ \, X \in \mathbb{Z}^n \mid EX = O \, \}.
  \]
\end{definition}

We want to find a lattice basis of $L(E)$; that is, a set of vectors in $L(E)$ which
are linearly independent over $\mathbb{Q}$ and such that every vector in $L(E)$ is
an integer linear combination of these basis vectors.  To find a lattice basis, we need
to recall the definition of the Hermite normal form (HNF) of an integer matrix.

\begin{definition}
The $m \times n$ matrix $H$ over $\mathbb{Z}$ is in \textbf{Hermite normal form}
(HNF) if there exists an integer $r$ (the rank of $H$) with $0 \le r \le m$ and a
sequence of integers $1 \le j_1 < j_2 < \cdots < j_r \le n$ such that
  \begin{enumerate}
  \item
  $H_{ij} = 0$ for $1 \le i \le r$ and $1 \le j < j_i$,
  \item
  $H_{i,j_i} \ge 1$ for $1 \le i \le r$,
  \item
  $0 \le H_{k,j_i} < H_{i,j_i}$ for $1 \le i \le r$ and $1 \le k < i$,
  \item
  $H_{ij} = 0$ for $r+1 \le i \le m$ and $1 \le j \le n$.
  \end{enumerate}
\end{definition}

The following two results are Lemmas 19 and 20 in Bremner and Peresi
\cite{BremnerPeresi}.

\begin{lemma}
If $E$ is an $m \times n$ matrix over $\mathbb{Z}$, then there is a unique $m \times n$
matrix $H$ over $\mathbb{Z}$ in HNF such that $UE = H$ for some $m \times m$ matrix
$U$ over $\mathbb{Z}$ with $\det(U) = \pm 1$. (The matrix $U$ is not unique.)
\end{lemma}

\begin{lemma}
Let $E$ be an $m \times n$ matrix over $\mathbb{Z}$, let $H$ be the HNF of $E^t$ ,
and let $U$ be an $n \times n$ matrix over $\mathbb{Z}$ with $\det(U) = \pm 1$ and
$U E^t = H$. If $r$ is the rank of $H$, then the last $n-r$ rows of $U$ form a lattice
basis for $L(E)$.
\end{lemma}

  \begin{table}
  \begin{center}
  \[
  \left[
  \begin{array}{rrrrrrrrrrrr}
  1 & 0 & 0 & 1 & 0 & 0 & 0 & 0 & 3 & 0 & 0 & 3 \\
  0 & 1 & 0 & 0 & 0 & 0 & 1 & 0 & 3 & 0 & 0 & 3 \\
  0 & 0 & 1 & 0 & 0 & 0 & 0 & 0 & 3 & 1 & 0 & 3 \\
  0 & 0 & 0 & 0 & 1 & 0 & 0 & 1 & 3 & 0 & 0 & 3 \\
  0 & 0 & 0 & 0 & 0 & 1 & 0 & 0 & 3 & 0 & 1 & 3 \\
  0 & 0 & 0 & 0 & 0 & 0 & 0 & 0 & 4 & 0 & 0 & 4 \\
  0 & 0 & 0 & 0 & 0 & 0 & 0 & 0 & 0 & 0 & 0 & 0 \\
  \vdots & \vdots & \vdots & \vdots & \vdots & \vdots &
  \vdots & \vdots & \vdots & \vdots & \vdots &\vdots
  \end{array}
  \right]
  \]
  \end{center}
  \caption{Hermite normal form $H$ of $E^t$ for $n = 2$}
  \label{binaryhermiteform}
  \begin{center}
  \[
  \left[
  \begin{array}{rrrrrrrrrrrrrrr}
  2 &  1 & -2 &  1 & -1 &  0 &  0 & 0 & 0 & 0 & 0 & 0 & 0 & 0 & 0 \\
  0 &  2 &  0 &  0 & -1 &  0 &  0 & 0 & 0 & 0 & 0 & 0 & 0 & 0 & 0 \\
  2 &  0 & -1 &  0 &  0 &  0 &  0 & 0 & 0 & 0 & 0 & 0 & 0 & 0 & 0 \\
  4 & -2 & -2 &  1 &  2 & -1 & -1 & 0 & 0 & 0 & 0 & 0 & 0 & 0 & 0 \\
  3 & -1 & -1 &  1 &  1 & -1 & -1 & 0 & 0 & 0 & 0 & 0 & 0 & 0 & 0 \\
  4 & -1 & -2 &  1 &  1 & -1 & -1 & 0 & 0 & 0 & 0 & 0 & 0 & 0 & 0 \\
  1 & -1 & -1 &  1 &  1 & -1 &  0 & 0 & 0 & 0 & 0 & 0 & 0 & 0 & 0 \\
  6 & -6 & -2 &  2 &  5 & -3 & -3 & 1 & 0 & 0 & 0 & 0 & 0 & 0 & 0 \\
  4 & -4 & -2 &  1 &  2 & -1 & -1 & 0 & 1 & 0 & 0 & 0 & 0 & 0 & 0 \\
  8 & -8 & -3 &  2 &  6 & -3 & -3 & 0 & 0 & 1 & 0 & 0 & 0 & 0 & 0 \\
  1 & -2 &  1 & -1 &  2 & -1 & -1 & 0 & 0 & 0 & 1 & 0 & 0 & 0 & 0 \\
  6 & -7 & -1 &  1 &  6 & -3 & -3 & 0 & 0 & 0 & 0 & 1 & 0 & 0 & 0 \\
  3 & -2 & -2 &  1 &  1 & -1 & -1 & 0 & 0 & 0 & 0 & 0 & 1 & 0 & 0 \\
  0 & -1 &  1 &  0 &  1 & -1 & -1 & 0 & 0 & 0 & 0 & 0 & 0 & 1 & 0 \\
  3 & -4 & -1 &  0 &  3 & -1 & -1 & 0 & 0 & 0 & 0 & 0 & 0 & 0 & 1
  \end{array}
  \right]
  \]
  \end{center}
  \caption{Transform matrix $U$ with $U E^t = H$ for $n = 2$}
  \label{binarytransformoriginal}
  \end{table}

For the matrix $E$ of Table \ref{binaryexpansionmatrix}, method (b) implemented with
the Maple command
  \[
  \texttt{HermiteForm( Transpose(E), output=['H','U'] ):}
  \]
gives the Hermite normal form $H$ in Table \ref{binaryhermiteform} and the transform
matrix $U$ in Table \ref{binarytransformoriginal}.  Since the rank of $E$ is 6, the last 9
rows of $U$ form a lattice basis of the integer nullspace of $E^t$.  The squared norms
of these 9 row vectors are 11, 5, 5, 31, 15, 25, 6, 124, 44, 196, 14, 142, 22, 6, 38.  At this
point the lengths of the basis vectors are greater than those obtained with the RCF;
however, we know that we have an integer basis of the nullspace lattice.

  \begin{table}
  \begin{center}
  \[
  \left[
  \begin{array}{rrrrrrrrrrrrrrr}
   1 &  1 &  0 &  0 & -1 &  0 & -1 &  0 &  0 & 0 &  0 &  0 &  0 &  1 & 0 \\
   1 &  0 &  0 &  0 &  0 &  0 &  0 &  0 &  1 & 0 &  0 &  0 & -1 &  0 & 0 \\
   0 &  1 &  0 &  0 &  0 &  0 &  1 &  0 &  0 & 0 &  0 &  0 & -1 &  0 & 0 \\
   1 &  0 &  0 &  0 &  1 &  0 &  0 &  0 &  0 & 0 &  0 &  0 & -1 &  0 & 0 \\
   0 &  1 &  1 &  0 &  0 &  0 &  0 &  0 &  0 & 0 &  0 &  0 & -1 &  0 & 0 \\
   1 &  1 &  0 &  0 &  0 &  0 &  0 &  0 &  0 & 0 &  0 &  0 & -1 &  0 & 0 \\
   1 & -1 & -1 &  1 &  1 & -1 &  0 &  0 &  0 & 0 &  0 &  0 &  0 &  0 & 0 \\
  -1 &  1 &  0 &  0 &  0 &  0 &  1 & -1 & -1 & 1 &  0 &  0 &  0 &  0 & 0 \\
   1 &  0 & -1 &  0 &  0 &  0 &  0 & -1 & -1 & 1 &  0 &  0 &  1 &  0 & 0 \\
   0 & -1 &  0 &  0 &  1 &  0 &  1 & -1 &  0 & 1 &  0 &  0 & -1 &  0 & 0 \\
   0 &  0 & -1 &  0 &  0 & -1 &  1 & -1 &  0 & 0 &  1 &  0 &  1 &  0 & 0 \\
   0 &  0 &  0 &  0 &  1 & -1 &  0 & -1 & -1 & 0 &  0 &  1 &  1 &  0 & 0 \\
   1 & -1 &  0 &  0 &  0 &  0 & -1 &  0 &  0 & 1 &  0 & -1 &  0 &  1 & 0 \\
   0 &  1 &  0 & -1 &  0 &  0 &  0 &  0 & -1 & 0 &  1 &  0 &  1 & -1 & 0 \\
   1 & -1 & -1 &  0 &  0 &  1 &  0 &  0 &  0 & 0 & -1 &  0 &  0 &  0 & 1
  \end{array}
  \right]
  \]
  \end{center}
  \caption{Reduced transform matrix $U$ with $U E^t = H$ for $n = 2$}
  \label{binarytransformreduced}
  \end{table}

To improve these results we use the LLL algorithm for lattice basis reduction, following
Bremner and Peresi \cite{BremnerPeresi}. For the matrix $E$ of Table
\ref{binaryexpansionmatrix}, the Maple command
  \[
  \texttt{HermiteForm( Transpose(E), output='U', method='integer[reduced]' ):}
  \]
produces the transform matrix $U$ of Table \ref{binarytransformreduced}. The
bottom 9 rows of $U$ are a reduced basis for the integral nullspace lattice of $E$;
every one of these vectors has squared norm 6, the same as the shortest vector
from Table \ref{binarybasis}. Even for this small matrix, the LLL algorithm has
produced a basis of the nullspace with significantly shorter vectors. The
identity corresponding to row 9 of $U$ generates the entire nullspace:
  \begin{equation}
  \label{reducedidentity}
  \begin{array}{l}
  \opleft \opleft \opleft a, b \opright, c \opright, d \opright
  - \opleft \opleft \opleft a, c \opright, b \opright, d \opright
  - \opleft \opleft \opleft b, c \opright, d \opright, a \opright
  - \opleft \opleft \opleft b, d \opright, a \opright, c \opright
  \\[4pt]
  {}
  + \opleft \opleft \opleft b, d \opright, c \opright, a \opright
  + \opleft \opleft a, b \opright, \opleft c, d \opright \opright
  =
  0.
  \end{array}
  \end{equation}
Identity \eqref{reducedidentity} has one more term than the binary
recombination identity \eqref{recombination}, but its coefficient vector is
shorter.

Identities \eqref{recombination} and \eqref{reducedidentity} both have only one
term in the second association type, so both imply that every monomial in the
second type can be expressed as a linear combination of monomials in the first
type. For example, identity \eqref{reducedidentity} can be written in the form
  \begin{align*}
  \opleft \opleft a, b \opright, \opleft c, d \opright \opright
  &=
  - \, \opleft \opleft \opleft a, b \opright, c \opright, d \opright
  + \opleft \opleft \opleft a, c \opright, b \opright, d \opright
  + \opleft \opleft \opleft b, c \opright, d \opright, a \opright
  \\
  &\quad
  + \opleft \opleft \opleft b, d \opright, a \opright, c \opright
  - \opleft \opleft \opleft b, d \opright, c \opright, a \opright.
  \end{align*}
This identity can be used to reduce the number of association types that we
need to consider when studying identities of higher degree.

\section{Ternary intermolecular recombination}

We now present the main results of this paper: a complete and minimal set of
polynomial identities of degree $\le 9$ for ternary intermolecular
recombination.

\subsection{Degree 3}
Every polynomial identity in degree 3 satisfied by ternary intermolecular
recombination is a consequence of the complete symmetry identity. Indeed,
equation \eqref{completesymmetry} implies that there is only one monomial in
degree 3, namely $\{ a, b, c \}$, and hence the only possibly polynomial
identity in degree 3 is $\{ a, b, c \} = 0$ for all $a, b, c$; this clearly
does not hold.

\subsection{Degree 5}
In degree 5 there are three distinct association types for a ternary operation:
  \[
  \opleft \opleft a,b,c \opright,d,e \opright,
  \quad
  \opleft a,\opleft b,c,d \opright,e \opright,
  \quad
  \opleft a,b,\opleft c,d,e \opright \opright.
  \]
Equation \eqref{completesymmetry} implies that we need to consider only the
first type, and that we need to consider only 10 distinct monomials in this
type; we order these monomials lexicographically:
  \begin{alignat*}{4}
  &\opleft \opleft a,b,c \opright,d,e \opright, &\quad
  &\opleft \opleft a,b,d \opright,c,e \opright, &\quad
  &\opleft \opleft a,b,e \opright,c,d \opright, &\quad
  &\opleft \opleft a,c,d \opright,b,e \opright,
  \\
  &\opleft \opleft a,c,e \opright,b,e \opright, &\quad
  &\opleft \opleft a,d,e \opright,b,c \opright, &\quad
  &\opleft \opleft b,c,d \opright,a,e \opright, &\quad
  &\opleft \opleft b,c,e \opright,a,d \opright,
  \\
  &\opleft \opleft b,d,e \opright,a,c \opright, &\quad
  &\opleft \opleft c,d,e \opright,a,b \opright.
  \end{alignat*}
Any polynomial identity in degree 5 satisfied by ternary intermolecular
recombination is a linear combination of these 10 monomials.

\begin{lemma} \label{degree5}
Every polynomial identity of degree 5 satisfied by ternary intermolecular
recombination is a consequence of the complete symmetry identity in degree 3.
\end{lemma}

\begin{proof}
Consider the expansions of the 10 monomials in degree 5 using two applications
of ternary intermolecular recombination. For example, since
  \begin{align*}
  \opleft a,b,c \opright
  &=
  (a_1,b_2,c_3) + (a_1,c_2,b_3) + (b_1,a_2,c_3)
  \\
  &\quad
  +
  (b_1,c_2,a_3) + (c_1,a_2,b_3) + (c_1,b_2,a_3),
  \end{align*}
we obtain
  \begin{align*}
  \opleft \opleft a,b,c \opright,d,e \opright
  &=
  2 (a_1,d_2,e_3) + 2 (a_1,e_2,d_3) + 2 (d_1,a_2,e_3) + 2 (d_1,e_2,a_3)
  \\
  &\quad
  +
  2 (e_1,a_2,d_3) + 2 (e_1,d_2,a_3) + 2 (b_1,d_2,e_3) + 2 (b_1,e_2,d_3)
  \\
  &\quad
  +
  2 (d_1,b_2,e_3) + 2 (d_1,e_2,b_3) + 2 (e_1,b_2,d_3) + 2 (e_1,d_2,b_3)
  \\
  &\quad
  +
  2 (c_1,d_2,e_3) + 2 (c_1,e_2,d_3) + 2 (d_1,c_2,e_3) + 2 (d_1,e_2,c_3)
  \\
  &\quad
  +
  2 (e_1,c_2,d_3) + 2 (e_1,d_2,c_3).
  \end{align*}
Each term in these 10 expansions has the form $( x_1, y_2, z_3 )$ where
$(x,y,z)$ is a permutation of a 3-element subset of $\{ a, b, c, d, e \}$.
There are 60 such triples, which we order lexicographically:
  \[
  \begin{array}{cccccc}
  (a_1,b_2,c_3), &
  (a_1,b_2,d_3), &
  (a_1,b_2,e_3), &
  (a_1,c_2,b_3), &
  (a_1,c_2,d_3), &
  (a_1,c_2,e_3),
  \\
  \vdots & \vdots & \vdots & \vdots & \vdots & \vdots
  \\
  (e_1,c_2,a_3), &
  (e_1,c_2,b_3), &
  (e_1,c_2,d_3), &
  (e_1,d_2,a_3), &
  (e_1,d_2,b_3), &
  (e_1,d_2,c_3).
  \end{array}
  \]
We construct the $60 \times 10$ matrix $E$ in which entry $(i,j)$ contains the
coefficient of triple $i$ in the expansion of monomial $j$. The polynomial
identities in degree 5 for ternary intermolecular recombination correspond to
the nonzero vectors in the nullspace of $E$. It suffices to show that this
matrix has full rank, and for this it suffices to find 10 rows for which the
corresponding $10 \times 10$ submatrix has full rank. Rows 1, 2, 3, 5, 6, 9,
17, 18, 21 and 33 produce the submatrix displayed in Table
\ref{degree5submatrix}; the row canonical form of this submatrix is the
identity matrix.
\end{proof}

  \begin{table}
  \begin{center}
  \[
  \begin{array}{rc|rrrrrrrrrr}
  \text{row} & \text{triple} & \multicolumn{10}{c}{\text{matrix entries}} \\ \hline
  &&&&&&&&&&& \\[-10pt]
   1 & ( a_1, b_2, c_3 ) & 0 & 0 & 0 & 0 & 0 & 2 & 0 & 0 & 2 & 2 \\
   2 & ( a_1, b_2, d_3 ) & 0 & 0 & 0 & 0 & 2 & 0 & 0 & 2 & 0 & 2 \\
   3 & ( a_1, b_2, e_3 ) & 0 & 0 & 0 & 2 & 0 & 0 & 2 & 0 & 0 & 2 \\
   5 & ( a_1, c_2, d_3 ) & 0 & 0 & 2 & 0 & 0 & 0 & 0 & 2 & 2 & 0 \\
   6 & ( a_1, c_2, e_3 ) & 0 & 2 & 0 & 0 & 0 & 0 & 2 & 0 & 2 & 0 \\
   9 & ( a_1, d_2, e_3 ) & 2 & 0 & 0 & 0 & 0 & 0 & 2 & 2 & 0 & 0 \\
  17 & ( b_1, c_2, d_3 ) & 0 & 0 & 2 & 0 & 2 & 2 & 0 & 0 & 0 & 0 \\
  18 & ( b_1, c_2, e_3 ) & 0 & 2 & 0 & 2 & 0 & 2 & 0 & 0 & 0 & 0 \\
  21 & ( b_1, d_2, e_3 ) & 2 & 0 & 0 & 2 & 2 & 0 & 0 & 0 & 0 & 0 \\
  33 & ( c_1, d_2, e_3 ) & 2 & 2 & 2 & 0 & 0 & 0 & 0 & 0 & 0 & 0
  \end{array}
  \]
  \end{center}
  \caption{Submatrix for the proof of Lemma \ref{degree5}}
  \label{degree5submatrix}
  \end{table}

\subsection{Degree 7}
We now consider 7 distinct molecules, each divided into 3 submolecules:
  \begin{alignat*}{4}
  a &= ( a_1, a_2, a_3 ), &\quad
  b &= ( b_1, b_2, b_3 ), &\quad
  c &= ( c_1, c_2, c_3 ), &\quad
  d &= ( d_1, d_2, d_3 ), \\
  e &= ( e_1, e_2, e_3 ), &\quad
  f &= ( f_1, f_2, f_3 ), &\quad
  g &= ( g_1, g_2, g_3 ).
  \end{alignat*}
After performing three applications of ternary intermolecular recombination on
these molecules with some order of operations and some permutation of
molecules, we obtain a linear combination of triples of the form $( x_1, y_2,
z_3 )$ where $( x, y, z )$ is a permutation of a 3-element subset of $\{ a, b,
c, d, e, f, g \}$. The number of triples which can occur as terms in these
linear combinations is therefore
  \[
  3! \binom73 = 210.
  \]
We order these triples lexicographically by the permutation $( x, y, z )$.

Equation \eqref{completesymmetry} implies that we need only two association
types in degree 7,
  \[
  \opleft \opleft \opleft a, b, c \opright, d, e \opright, f, g \opright
  \quad \text{and} \quad
  \opleft \opleft a, b, c \opright, \opleft d, e, f \opright, g \opright,
  \]
which contain respectively the following numbers of distinct monomials,
  \[
  \frac{7!}{3!2!2!} = 210
  \quad \text{and} \quad
  \frac{1}{2!} \cdot \frac{7!}{3!3!1!} = 70.
  \]
Within each association type, we order the monomials lexicographically by the
permutation of the variables. Any polynomial identity in degree 7 satisfied by
ternary intermolecular recombination is a linear combination of these 280
nonassociative monomials.

The expansion matrix $E$ in degree 7 has 210 rows and 280 columns; entry
$(i,j)$ is the coefficient of triple $i$ in the expansion of nonassociative
monomial $j$. The expansions of the two association types are displayed in
Tables \ref{degree7type1expansion} and \ref{degree7type2expansion}. Each
expansion has $6^3 = 216$ terms; after collecting terms, the first type
produces a linear combination of 30 triples with multiplicities 4 and 12, and
the second type produces a linear combination of 54 triples with multiplicity
4. We use Maple to find that the matrix $E$ has rank 35, and so its nullspace
has dimension 245. We compute the row canonical form of $E$, and find that
every entry of the RCF is an integer. We obtain the canonical basis of the
nullspace of $E$; this is a list of 245 vectors of dimension 280 with integer
components. We sort these vectors by increasing Euclidean norm; the list of
squared norms is displayed in Table \ref{degree7canonicalnorms}.

  \begin{table}
  \begin{alignat*}{5}
          4 [ a_1, f_2, g_3 ] &
  \,+\,&  4 [ a_1, g_2, f_3 ] &
  \,+\,&  4 [ b_1, f_2, g_3 ] &
  \,+\,&  4 [ b_1, g_2, f_3 ] &
  \,+\,&  4 [ c_1, f_2, g_3 ] \\
  \,+\,   4 [ c_1, g_2, f_3 ] &
  \,+\,& 12 [ d_1, f_2, g_3 ] &
  \,+\,& 12 [ d_1, g_2, f_3 ] &
  \,+\,& 12 [ e_1, f_2, g_3 ] &
  \,+\,& 12 [ e_1, g_2, f_3 ] \\
  \,+\,   4 [ f_1, a_2, g_3 ] &
  \,+\,&  4 [ f_1, b_2, g_3 ] &
  \,+\,&  4 [ f_1, c_2, g_3 ] &
  \,+\,& 12 [ f_1, d_2, g_3 ] &
  \,+\,& 12 [ f_1, e_2, g_3 ] \\
  \,+\,   4 [ f_1, g_2, a_3 ] &
  \,+\,&  4 [ f_1, g_2, b_3 ] &
  \,+\,&  4 [ f_1, g_2, c_3 ] &
  \,+\,& 12 [ f_1, g_2, d_3 ] &
  \,+\,& 12 [ f_1, g_2, e_3 ] \\
  \,+\,   4 [ g_1, a_2, f_3 ] &
  \,+\,&  4 [ g_1, b_2, f_3 ] &
  \,+\,&  4 [ g_1, c_2, f_3 ] &
  \,+\,& 12 [ g_1, d_2, f_3 ] &
  \,+\,& 12 [ g_1, e_2, f_3 ] \\
  \,+\,   4 [ g_1, f_2, a_3 ] &
  \,+\,&  4 [ g_1, f_2, b_3 ] &
  \,+\,&  4 [ g_1, f_2, c_3 ] &
  \,+\,& 12 [ g_1, f_2, d_3 ] &
  \,+\,& 12 [ g_1, f_2, e_3 ]
  \end{alignat*}
  \caption{Expansion of monomial
  $\opleft \opleft \opleft a, b, c \opright, d, e \opright, f, g \opright$}
  \label{degree7type1expansion}
  \begin{alignat*}{5}
         4 [ a_1, d_2, g_3 ] &
  \,+\,& 4 [ a_1, e_2, g_3 ] &
  \,+\,& 4 [ a_1, f_2, g_3 ] &
  \,+\,& 4 [ a_1, g_2, d_3 ] &
  \,+\,& 4 [ a_1, g_2, e_3 ] \\
  \,+\,  4 [ a_1, g_2, f_3 ] &
  \,+\,& 4 [ b_1, d_2, g_3 ] &
  \,+\,& 4 [ b_1, e_2, g_3 ] &
  \,+\,& 4 [ b_1, f_2, g_3 ] &
  \,+\,& 4 [ b_1, g_2, d_3 ] \\
  \,+\,  4 [ b_1, g_2, e_3 ] &
  \,+\,& 4 [ b_1, g_2, f_3 ] &
  \,+\,& 4 [ c_1, d_2, g_3 ] &
  \,+\,& 4 [ c_1, e_2, g_3 ] &
  \,+\,& 4 [ c_1, f_2, g_3 ] \\
  \,+\,  4 [ c_1, g_2, d_3 ] &
  \,+\,& 4 [ c_1, g_2, e_3 ] &
  \,+\,& 4 [ c_1, g_2, f_3 ] &
  \,+\,& 4 [ d_1, a_2, g_3 ] &
  \,+\,& 4 [ d_1, b_2, g_3 ] \\
  \,+\,  4 [ d_1, c_2, g_3 ] &
  \,+\,& 4 [ d_1, g_2, a_3 ] &
  \,+\,& 4 [ d_1, g_2, b_3 ] &
  \,+\,& 4 [ d_1, g_2, c_3 ] &
  \,+\,& 4 [ e_1, a_2, g_3 ] \\
  \,+\,  4 [ e_1, b_2, g_3 ] &
  \,+\,& 4 [ e_1, c_2, g_3 ] &
  \,+\,& 4 [ e_1, g_2, a_3 ] &
  \,+\,& 4 [ e_1, g_2, b_3 ] &
  \,+\,& 4 [ e_1, g_2, c_3 ] \\
  \,+\,  4 [ f_1, a_2, g_3 ] &
  \,+\,& 4 [ f_1, b_2, g_3 ] &
  \,+\,& 4 [ f_1, c_2, g_3 ] &
  \,+\,& 4 [ f_1, g_2, a_3 ] &
  \,+\,& 4 [ f_1, g_2, b_3 ] \\
  \,+\,  4 [ f_1, g_2, c_3 ] &
  \,+\,& 4 [ g_1, a_2, d_3 ] &
  \,+\,& 4 [ g_1, a_2, e_3 ] &
  \,+\,& 4 [ g_1, a_2, f_3 ] &
  \,+\,& 4 [ g_1, b_2, d_3 ] \\
  \,+\,  4 [ g_1, b_2, e_3 ] &
  \,+\,& 4 [ g_1, b_2, f_3 ] &
  \,+\,& 4 [ g_1, c_2, d_3 ] &
  \,+\,& 4 [ g_1, c_2, e_3 ] &
  \,+\,& 4 [ g_1, c_2, f_3 ] \\
  \,+\,  4 [ g_1, d_2, a_3 ] &
  \,+\,& 4 [ g_1, d_2, b_3 ] &
  \,+\,& 4 [ g_1, d_2, c_3 ] &
  \,+\,& 4 [ g_1, e_2, a_3 ] &
  \,+\,& 4 [ g_1, e_2, b_3 ] \\
  \,+\,  4 [ g_1, e_2, c_3 ] &
  \,+\,& 4 [ g_1, f_2, a_3 ] &
  \,+\,& 4 [ g_1, f_2, b_3 ] &
  \,+\,& 4 [ g_1, f_2, c_3 ] &
  \end{alignat*}
  \caption{Expansion of monomial
  $\opleft \opleft a, b, c \opright, \opleft d, e, f \opright, g \opright$}
  \label{degree7type2expansion}
  \end{table}

  \begin{table}
  \begin{tabular}{rrrrrrrrrr}
     4 &    4 &    4 &    4 &    4 &    4 &    4 &    4 &    6 &    6 \\
     6 &    6 &    6 &    6 &    6 &    6 &    6 &    6 &    6 &    6 \\
     6 &    6 &    6 &    6 &    6 &    6 &    6 &    6 &    6 &    6 \\
     6 &    8 &    8 &    8 &    8 &    8 &    8 &    8 &    8 &    8 \\
    10 &   10 &   10 &   10 &   10 &   10 &   10 &   10 &   10 &   10 \\
    10 &   10 &   12 &   12 &   12 &   12 &   12 &   14 &   16 &   32 \\
    32 &   34 &   34 &   34 &   34 &   52 &   52 &   52 &   52 &   52 \\
    52 &   80 &   80 &   82 &   82 &   86 &   86 &   88 &   94 &   94 \\
    96 &   96 &   98 &  100 &  100 &  100 &  100 &  100 &  100 &  104 \\
   104 &  104 &  108 &  108 &  110 &  110 &  110 &  110 &  110 &  110 \\
   118 &  118 &  122 &  122 &  128 &  128 &  134 &  134 &  134 &  134 \\
   136 &  138 &  138 &  140 &  140 &  144 &  146 &  146 &  146 &  152 \\
   152 &  156 &  156 &  156 &  156 &  158 &  158 &  158 &  158 &  158 \\
   162 &  162 &  162 &  170 &  170 &  170 &  170 &  174 &  174 &  174 \\
   176 &  176 &  178 &  178 &  182 &  184 &  186 &  188 &  188 &  188 \\
   188 &  190 &  190 &  190 &  190 &  198 &  206 &  210 &  216 &  228 \\
   232 &  234 &  240 &  240 &  242 &  246 &  246 &  246 &  246 &  246 \\
   252 &  254 &  256 &  266 &  270 &  728 &  728 &  744 &  744 &  756 \\
   756 &  876 &  876 & 1210 & 1210 & 1216 & 1216 & 1244 & 1244 & 1350 \\
  1386 & 1386 & 1402 & 1402 & 1408 & 1412 & 1414 & 1414 & 1440 & 1470 \\
  1482 & 1484 & 1494 & 1516 & 1568 & 1592 & 1678 & 1758 & 1818 & 1824 \\
  1852 & 1894 & 1908 & 1910 & 1918 & 1946 & 1946 & 1982 & 2002 & 2026 \\
  2064 & 2174 & 2196 & 2226 & 2246 & 2272 & 2312 & 2324 & 2326 & 2340 \\
  2346 & 2366 & 2390 & 2396 & 2402 & 2418 & 2430 & 2472 & 2490 & 2562 \\
  2634 & 2640 & 2654 & 2668 & 2844
  \end{tabular}
  \smallskip
  \caption{Squared norms of canonical basis vectors in degree 7}
  \label{degree7canonicalnorms}
  \end{table}

  \begin{table}
  \begin{tabular}{rrrrrrrrrrrrrrr}
   4 &  4 &  4 &  4 &  4 &  4 &  4 &  4 &  4 &  4 &  4 &  4 &  4 &  4 &  4 \\
   4 &  4 &  4 &  4 &  4 &  4 &  4 &  4 &  4 &  4 &  4 &  4 &  4 &  4 &  4 \\
   4 &  4 &  4 &  4 &  4 &  4 &  4 &  4 &  4 &  6 &  6 &  6 &  6 &  6 &  6 \\
   6 &  6 &  6 &  6 &  6 &  6 &  6 &  6 &  6 &  6 &  6 &  6 &  6 &  6 &  6 \\
   6 &  6 &  6 &  6 &  6 &  6 &  6 &  6 &  6 &  6 &  6 &  6 &  6 &  8 &  8 \\
   8 &  8 &  8 &  8 &  8 &  8 &  8 &  8 &  8 &  8 &  8 &  8 &  8 &  8 &  8 \\
   8 &  8 &  8 &  8 &  8 &  8 &  8 &  8 &  8 &  8 &  8 &  8 &  8 &  8 &  8 \\
   8 &  8 &  8 &  8 &  8 &  8 &  8 &  8 &  8 &  8 &  8 &  8 & 10 & 10 & 10 \\
  10 & 10 & 10 & 10 & 10 & 10 & 10 & 10 & 12 & 12 & 12 & 12 & 12 & 12 & 12 \\
  12 & 12 & 12 & 12 & 12 & 12 & 12 & 12 & 12 & 12 & 12 & 12 & 12 & 12 & 12 \\
  12 & 12 & 12 & 12 & 12 & 12 & 12 & 14 & 14 & 14 & 14 & 14 & 14 & 14 & 14 \\
  14 & 14 & 14 & 14 & 14 & 14 & 14 & 16 & 16 & 16 & 16 & 16 & 16 & 16 & 16 \\
  16 & 16 & 16 & 16 & 16 & 16 & 16 & 16 & 18 & 18 & 18 & 18 & 18 & 18 & 18 \\
  18 & 20 & 20 & 20 & 20 & 20 & 20 & 20 & 20 & 20 & 20 & 20 & 22 & 22 & 22 \\
  22 & 22 & 22 & 22 & 22 & 22 & 24 & 24 & 24 & 24 & 24 & 24 & 24 & 24 & 24 \\
  26 & 26 & 26 & 26 & 26 & 28 & 28 & 28 & 28 & 28 & 28 & 30 & 30 & 30 & 32 \\
  32 & 34 & 34 & 34 & 38 &
  \end{tabular}
  \smallskip
  \caption{Squared norms of reduced basis vectors in degree 7}
  \label{degree7reducednorms}
  \end{table}

We now perform further computations, using modular arithmetic to save memory,
to determine a set of generators for the nullspace as a module over the
symmetric group $S_7$. (We use $p = 101$, the smallest prime greater than 100;
any prime greater than the degree of the identities would produce the same
dimensions.) For each basis vector in the nullspace, we apply all 5040
permutations of $a$, $b$, $c$, $d$, $e$, $f$, $g$ to the corresponding
identity, and store the results in a matrix of size $5040 \times 280$. We
compute the row canonical form of this matrix; its nonzero rows form a basis of
the $S_7$-module generated by the identity, and its rank is the dimension of
this module. We process the nullspace basis vectors in order, saving previous
results (the nonzero rows which form a basis of the $S_7$-module generated by
the previous basis vectors). At each stage, an increase in the rank implies
that the current identity is a new generator of the nullspace as a module over
$S_7$. The results of these computations show that the canonical basis vectors
with positions 1, 10 and 60 (after the basis vectors have been sorted by
increasing Euclidean norm) represent polynomial identities which generate the
entire nullspace as an $S_7$-module:
  \allowdisplaybreaks
  \begin{align*}
  I
  &=
     \opleft \opleft \opleft a, b, c \opright, d, e \opright, f, g \opright
  -  \opleft \opleft \opleft a, b, e \opright, c, d \opright, f, g \opright
  -  \opleft \opleft \opleft b, c, d \opright, a, e \opright, f, g \opright
  \\
  &\quad
  +  \opleft \opleft \opleft b, d, e \opright, a, c \opright, f, g \opright,
  \\
  J
  &=
     \opleft \opleft \opleft a, b, e \opright, c, g \opright, d, f \opright
  -  \opleft \opleft \opleft a, b, e \opright, f, g \opright, c, d \opright
  -  \opleft \opleft \opleft a, c, e \opright, b, g \opright, d, f \opright
  \\
  &\quad
  +  \opleft \opleft \opleft a, c, e \opright, f, g \opright, b, d \opright
  +  \opleft \opleft \opleft a, e, f \opright, b, g \opright, c, d \opright
  -  \opleft \opleft \opleft a, e, f \opright, c, g \opright, b, d \opright,
  \\
  K
  &=
   2 \opleft \opleft \opleft a, b, c \opright, d, e \opright, f, g \opright
  -  \opleft \opleft \opleft a, b, c \opright, d, g \opright, e, f \opright
  -2 \opleft \opleft \opleft a, b, c \opright, e, g \opright, d, f \opright
  \\
  &\quad
  +2 \opleft \opleft \opleft a, b, d \opright, c, e \opright, f, g \opright
  -2 \opleft \opleft \opleft a, b, d \opright, c, g \opright, e, f \opright
  -  \opleft \opleft \opleft a, b, d \opright, e, g \opright, c, f \opright
  \\
  &\quad
  -  \opleft \opleft \opleft a, b, e \opright, c, d \opright, f, g \opright
  +  \opleft \opleft \opleft a, b, e \opright, c, g \opright, d, f \opright
  +3 \opleft \opleft \opleft a, b, g \opright, c, d \opright, e, f \opright
  \\
  &\quad
  -  \opleft \opleft \opleft a, c, d \opright, b, e \opright, f, g \opright
  -  \opleft \opleft \opleft b, c, d \opright, a, e \opright, f, g \opright
  +  \opleft \opleft a, b, e \opright, \opleft c, d, g \opright, f \opright.
  \end{align*}
Every identity in the nullspace is a linear combination of permutations of
these three identities, which have squared norms 4, 6 and 32; the second
identity has been normalized so that it begins with a positive coefficient.

We can get significantly better results using the Hermite normal form and
lattice basis reduction.

\begin{theorem} \label{degree7}
Every polynomial identity of degree 7 satisfied by ternary intermolecular
recombination is a consequence of complete symmetry in degree 3 and these three
identities in degree 7:
  \allowdisplaybreaks
  \begin{align*}
  P
  &=
  \opleft \opleft \opleft a, c, f \opright, e, g \opright, b, d \opright
  - \opleft \opleft \opleft a, c, g \opright, e, f \opright, b, d \opright
  - \opleft \opleft \opleft c, e, f \opright, a, g \opright, b, d \opright
  \\
  &\quad
  + \opleft \opleft \opleft c, e, g \opright, a, f \opright, b, d \opright,
  \\
  Q
  &=
  \opleft \opleft \opleft a, d, e \opright, b, g \opright, c, f \opright
  - \opleft \opleft \opleft a, d, e \opright, f, g \opright, b, c \opright
  - \opleft \opleft \opleft b, d, e \opright, a, g \opright, c, f \opright
  \\
  &\quad
  + \opleft \opleft \opleft b, d, e \opright, f, g \opright, a, c \opright
  + \opleft \opleft \opleft d, e, f \opright, a, g \opright, b, c \opright
  - \opleft \opleft \opleft d, e, f \opright, b, g \opright, a, c \opright,
  \\
  R
  &=
  \opleft \opleft \opleft a, b, c \opright, f, g \opright, d, e \opright
  + \opleft \opleft \opleft a, b, g \opright, d, f \opright, c, e \opright
  + \opleft \opleft \opleft a, c, d \opright, b, f \opright, e, g \opright
  \\
  &\quad
  - \opleft \opleft \opleft a, c, d \opright, b, g \opright, e, f \opright
  + \opleft \opleft \opleft a, c, d \opright, f, g \opright, b, e \opright
  - \opleft \opleft \opleft a, c, f \opright, b, d \opright, e, g \opright
  \\
  &\quad
  + \opleft \opleft \opleft a, c, g \opright, d, f \opright, b, e \opright
  - \opleft \opleft \opleft a, d, f \opright, c, g \opright, b, e \opright
  - \opleft \opleft \opleft a, d, g \opright, b, c \opright, e, f \opright
  \\
  &\quad
  - \opleft \opleft \opleft a, f, g \opright, b, c \opright, d, e \opright
  + \opleft \opleft \opleft b, c, d \opright, a, g \opright, e, f \opright
  - \opleft \opleft a, d, g \opright, \opleft b, c, f \opright, e \opright.
  \end{align*}
Identities $P$ and $Q$ are independent: neither implies the other. Identity $R$
implies both $P$ and $Q$, but identities $P$ and $Q$ together do not imply $R$.
\end{theorem}

\begin{proof}
We apply the Maple command
  \[
  \texttt{HermiteForm( Transpose(E), output='U', method='integer[reduced]' ):}
  \]
to the transpose (size $280 \times 210$) of the expansion matrix $E$, and
obtain a transform matrix $U$ (size $280 \times 280$). The bottom 245 rows of
$U$ are a basis for the integral nullspace lattice of $E$. We sort these
vectors by increasing Euclidean norm; the squared norms are displayed in Table
\ref{degree7reducednorms}. It is clear by comparing Tables
\ref{degree7canonicalnorms} and \ref{degree7reducednorms} that the basis
vectors obtained in this way are generally much shorter than those obtained
from the RCF. Identities $P$, $Q$ and $R$ have coefficients $\pm 1$ and the
squared norms of the corresponding coefficient vectors are 4, 6 and 12. (These
identities correspond to the reduced basis vectors in positions 1, 40 and 129
after the reduced basis vectors have been sorted by increasing Euclidean norm.)
These three identities can be verified by expanding each term using three
applications of ternary intermolecular recombination, and then checking that
the results collapse to zero.

To prove the stated dependence and independence relations we proceed as
follows. Given an identity $I(a,b,c,d,e,f,g)$ we apply all 5040 permutations of
$a$, $b$, $c$, $d$, $e$, $f$, $g$ and store the coefficients of the resulting
identities in a $5040 \times 280$ matrix. The rank of this matrix is the
dimension of the $S_7$-submodule of identities generated by $I$. To check
whether another identity $J(a,b,c,d,e,f,g)$ is implied by $I$ we simply
determine whether $J$ is in the row space of this matrix. We find that identity
$P$ produces dimension 105, and identity $Q$ produces dimension 127; stacking
the two corresponding matrices together, we see that these two identities
together produce dimension 155. However, identity $R$ produces dimension 245,
and this is also the rank produced by the three identities together. Since 245
is also the dimension of the nullspace of the expansion matrix, it follows that
every identity in degree 7 is a consequence of the identity $R$, in the sense
that every identity in degree 7 is a linear combination of permutations of
$R(a,b,c,d,e,f,g)$.
\end{proof}

The identity $R$ of Theorem \ref{degree7} will be called the \textbf{ternary
recombination identity}. The last term in $R$ is the only term in identities
$P$, $Q$ and $R$ which has the second association type, so we have the
following result.

\begin{corollary} \label{rewritesecondtype}
Every polynomial identity of degree $\le 7$ satisfied by ternary intermolecular
recombination is a consequence of complete symmetry and the following identity,
which expresses any monomial in the second association type as a linear
combination of monomials in the first association type:
  \begin{align*}
  &
  \opleft \opleft a, b, c \opright, \opleft d, e, f \opright, g \opright
  =
  \\
  &\;\;\;\;
  \opleft \opleft \opleft a, d, e \opright, f, c \opright, b, g \opright
  + \opleft \opleft \opleft a, d, c \opright, b, f \opright, e, g \opright
  + \opleft \opleft \opleft a, e, b \opright, d, f \opright, g, c \opright
  \\
  &
  - \opleft \opleft \opleft a, e, b \opright, d, c \opright, g, f \opright
  + \opleft \opleft \opleft a, e, b \opright, f, c \opright, d, g \opright
  - \opleft \opleft \opleft a, e, f \opright, d, b \opright, g, c \opright
  \\
  &
  + \opleft \opleft \opleft a, e, c \opright, b, f \opright, d, g \opright
  - \opleft \opleft \opleft a, b, f \opright, e, c \opright, d, g \opright
  - \opleft \opleft \opleft a, b, c \opright, d, e \opright, g, f \opright
  \\
  &
  - \opleft \opleft \opleft a, f, c \opright, d, e \opright, b, g \opright
  + \opleft \opleft \opleft d, e, b \opright, a, c \opright, g, f \opright.
  \end{align*}
(We have permuted the variables so that the term in the second association type
has the identity permutation.)
\end{corollary}

\subsection{Degree 9} The number of permutations of 3-element subsets of a set
of 9 variables is clearly
  \[
  3! \binom93 = 504.
  \]
This is the number of rows in the expansion matrix in degree 9.

Equation \eqref{completesymmetry} implies that we need only four association
types in degree 9:
  \begin{align*}
  &\opleft \opleft \opleft \opleft a, b, c \opright, d, e \opright, f, g \opright, h, i \opright
  \; \text{with} \;
  \frac{9!}{3!2!2!2!} = 7560 \; \text{monomials},
  \\
  &\opleft \opleft \opleft a, b, c \opright, \opleft d, e, f \opright, g \opright, h, i \opright
  \; \text{with} \;
  \frac{1}{2!} \cdot \frac{9!}{3!3!1!2!} = 2520 \; \text{monomials},
  \\
  &\opleft \opleft \opleft a, b, c \opright, d, e \opright, \opleft f, g, h \opright, i \opright
  \; \text{with} \;
  \frac{9!}{3!2!3!1!} = 5040 \; \text{monomials}, \text{and}
  \\
  &\opleft \opleft a, b, c \opright, \opleft d, e, f \opright, \opleft g, h, i \opright \opright
  \; \text{with} \;
  \frac{1}{3!} \cdot \frac{9!}{3!3!3!} = 280 \; \text{monomials}.
  \end{align*}
The total is 15400; this is the number of columns in the expansion matrix.

We initialize the expansion matrix and then use modular arithmetic (again with
$p = 101$) to compute its rank; the result is 84, so the dimension of the
nullspace is 15316. To decide whether any nullspace vector represents a new
identity in degree 9, we need to determine the dimension of the subspace of the
nullspace which consists of all consequences of the ternary recombination
identity $R = R(a,b,c,d,e,f,g)$ from Theorem \ref{degree7}. To obtain
consequences in degree 9, we can either
  \begin{enumerate}
  \item[($i$)]
replace one of the variables $x \in \{ a, b, c, d, e, f, g \}$ by the triple
$\{x,h,i\}$ where $h$ and $i$ are two new variables, or
  \item[($ii$)]
embed the entire identity $R$ in a triple $\{R,h,i\}$.
  \end{enumerate}
This gives 8 identities in degree 9 which generate the $S_9$-module of all
consequences of the identity $R$ in degree 7:
  \begin{alignat*}{2}
  R( \{a,h,i\}, b, c, d, e, f, g ),
  \qquad
  R( a, \{b,h,i\}, c, d, e, f, g ),
  \\
  R( a, b, \{c,h,i\}, d, e, f, g ),
  \qquad
  R( a, b, c, \{d,h,i\}, e, f, g ),
  \\
  R( a, b, c, d, \{e,h,i\}, f, g ),
  \qquad
  R( a, b, c, d, e, \{f,h,i\}, g ),
  \\
  R( a, b, c, d, e, f, \{g,h,i\} ),
  \qquad
  \{ R( a, b, c, d, e, f, g ), h, i \}.
  \end{alignat*}
We need to ``straighten'' the terms of these identities: that is, we use the
complete symmetry identity to replace each monomial by an equivalent monomial
in one of the four association types listed above.

We now apply all $9!$ permutations of $a$, $b$, $c$, $d$, $e$, $f$, $g$, $h$,
$i$ to each of the 8 consequences of $R$ in degree 9, store these permuted
identities in the rows of a matrix with 15400 columns, and compute the rank of
the matrix. For $\ell = 1, 2, \hdots, 8$ the dimension of the subspace
generated by the first $\ell$ consequences of $R$ is as follows:
  \[
  14071, \;
  15036, \;
  15036, \;
  15036, \;
  15316, \;
  15316, \;
  15316, \;
  15316.
  \]
Since 15316 is also the dimension of the nullspace of the expansion matrix, it
follows that every polynomial identity in degree 9 satisfied by ternary
intermolecular recombination is a consequence of the ternary recombination
identity in degree 7. Hence there are no new identities in degree 9.

\section{Conclusion}

The results of Bremner \cite{Bremner}, Sverchkov \cite{Sverchkov} and the
present paper are consistent with the following conjectures for the general
case of $n$-ary intermolecular recombination.

\begin{conjecture} \label{conjecture1}
Every polynomial identity satisfied by $n$-ary intermolecular recombination in
degree $2n{-}1$ (that is, in which each term involves two applications of the
$n$-ary operation) is a consequence of complete symmetry in degree $n$.
\end{conjecture}

\begin{conjecture} \label{conjecture2}
Every polynomial identity satisfied by $n$-ary intermolecular recombination in
degree $3n{-}2$ (that is, in which each term involves three applications of the
$n$-ary operation) is a consequence of complete symmetry in degree $n$ together
with an identity $R$ in degree $3n{-}2$ in which every coefficient equals $\pm
1$ and which implies that every monomial in the second association type,
  \[
  \opleft \,
  \opleft \, a_1, \hdots, a_n \, \opright,
  \opleft \, a_{n+1}, \hdots, a_{2n} \, \opright, \,
  a_{2n+1}, \hdots, a_{3n-2}
  \, \opright,
  \]
equals a linear combination of monomials in the first association type,
  \[
  \opleft \,
  \opleft \,
  \opleft \, a_1, \hdots, a_n \, \opright, \, a_{n+1}, \hdots, a_{2n-1} \, \opright, \,
  a_{2n}, \hdots, a_{3n-2} \,
  \opright.
  \]
\end{conjecture}

\begin{conjecture} \label{conjecture3}
Every polynomial identity satisfied by $n$-ary intermolecular recombination,
with no restriction on the degree, is a consequence of complete symmetry in
degree $n$ together with the identity $R$ from Conjecture \ref{conjecture2} in
degree $3n{-}2$.
\end{conjecture}

Resolving these conjectures is the most important open problem in the theory of
$n$-ary intermolecular recombination.

\begin{remark}
Shortly before the final version of this paper was sent to the editors, I received
Sverchkov's preprint \cite{Sverchkov2}, which contains a complete proof of all three
Conjectures for all $n$, with an explicit identity $R$ as required by Conjecture
\ref{conjecture2}.
\end{remark}

\section{Acknowledgements}

I thank the organizers (B. H. Li, R. C. Laubenbacher, J. J. P. Tian) of the
Special Session on Biomathematics at the First Joint Meeting of the Shanghai
Mathematical Society and the American Mathematical Society (Shanghai, China,
December 17--21, 2008) for the invitation to give a presentation at the
conference and to contribute a paper to this special issue.

This research was supported by NSERC, the Natural Sciences and Engineering
Research Council of Canada.

\end{document}